\documentclass[reqno,12pt]{amsart}

\title
[Lagrangian mean curvature flows in cotangent bundle]
{Generalized Lagrangian mean curvature flows: the cotangent bundle case}
\author[Knut Smoczyk]{\sc Knut Smoczyk$^\ast$}
\address{
$^\ast$Leibniz Universit\"at Hannover,
{F}akult\"at f\"ur Mathematik und Physik,
Institut f\"ur Mathematik,
Welfengarten 1,
30167 Hannover,
Germany}
\email{smoczyk@math.uni-hannover.de}
\author[Mao-Pei Tsui]{\sc Mao-Pei Tsui$^{\ast\ast}$}
\address{
$^{\ast\ast}$Department of Mathematics, National Taiwan University, Taipei 10617, Taiwan
 and
University of Toledo,
Department of Mathematics and Statistics,
2801 W. Bancroft St,
Toledo, Ohio 43606-3390}
\email{maopei@math.ntu.edu.tw, mao-pei.tsui@utoledo.edu}
\author[Mu-Tao Wang]{\sc Mu-Tao Wang$^{\ast\ast\ast}$}
\address{
$^{\ast\ast\ast}$Columbia University,
Department of Mathematics,
2990 Broadway,
New York, NY 10027}
\email{mtwang@math.columbia.edu}
\thanks{The first named author was supported by the DFG (German Research Foundation). 
The second named author  was partially  supported  by a  Collaboration Grant for Mathematicians from the Simons Foundation \#239677
and in part by Taiwan MOST grant 104-2115-M-002-001-MY2.
The third named author was partially supported by National Science
Foundation Grants DMS  1105483 and 1405152.}%

\date{June 3, 2015}

\usepackage{amscd,amsfonts,mathrsfs,amsthm,enumerate}
\usepackage{amssymb, amsmath}          
\usepackage{color}
\usepackage[shortalphabetic]{amsrefs}    

\newtheorem{thm}{Theorem}

\newtheorem{pro}{Proposition}
\newtheorem{lem}{Lemma}
\newtheorem{dfn}{Definition}

\newtheorem{rem}{Remark}

\numberwithin{cor}{section}
\numberwithin{pro}{section} \numberwithin{dfn}{section}
\numberwithin{lem}{section}
\numberwithin{rem}{section}\numberwithin{equation}{section}

\def\real     #1{{\mathbb R^{#1}}}

\def\hn {{\widehat\nabla}}
\def\hT {{\widehat T}}

\def\hR {{\widehat R}}

\def\hH {{\widehat H}}

\def\hrho{{\widehat\rho}}

\def\dd       #1#2#3{{#1}_{#2#3}}

\def\uddd     #1#2#3#4#5{#1^#2_{\phantom{#2}#3#4#5}}

\newcommand{\C}{\mathbb C}

\begin{document}

\maketitle
\begin{abstract}

In \cite{sw2}, we defined a generalized mean curvature
vector field on any almost Lagrangian submanifold with respect to a torsion connection on an almost K\"ahler manifold. The short time existence of the corresponding parabolic flow was established. In addition, it was shown that the flow preserves the Lagrangian condition as long as the connection satisfies an Einstein condition. In this article, we show that the canonical
 connection on the cotangent bundle of any Riemannian manifold is an Einstein connection (in fact, Ricci flat). The generalized mean curvature vector on any Lagrangian submanifold is related to the Lagrangian angle defined by the phase of a parallel $(n, 0)$ form, just like the Calabi-Yau case. We also show that the corresponding Lagrangian mean curvature flow in cotangent bundles preserves the exactness and  the zero Maslov class conditions. At the end, we prove a long time existence and convergence result to demonstrate the stability of the zero section of the cotangent bundle of spheres.
\end{abstract}

\section{Introduction}
An almost K\"ahler manifold $(N,\omega, J)$ is a symplectic manifold $(N, \omega)$ with an almost complex structure $J$ such that
$g=\langle\cdot,\cdot\rangle=\omega(\cdot, J\cdot)$ becomes a Riemannian metric. Any symplectic manifold admits an almost K\"ahler structure. In particular, on the cotangent bundle $N:=T^*\Sigma$ of a Riemannian manifold $(\Sigma,\sigma)$, there is a canonical almost K\"ahler structure with respect to the base metric $\sigma$. The associated metric $g$ on $N$ (see Proposition \ref{metric}) is in general not K\"ahler and the associated almost complex structure $J$ is in general not integrable.  In addition, there is a connection $\hn$ with torsion on the tangent bundle of $N$ which is both metric and complex, and the horizontal and vertical distributions are parallel with respect to this connection. The torsion of this connection is completely determined by the Riemannian curvature tensor of the base manifold $(\Sigma,\sigma)$. In \cite{sw2}, we defined a notion called Einstein connection (see Definition \ref{einstein connection}) for a complex and metric connection on an almost K\"ahler manifold. In the article, we
show:
\begin{thm}\label{theo 1}
Let $(\Sigma,\sigma)$ be a Riemannian manifold and $(J, \omega, g)$ be the almost K\"ahler structure defined on the cotangent
bundle $N=T^*\Sigma$ with the canonical connection $\hn$ (see $\S2$). Then the  Ricci form $\hrho$ of  $\hn$ vanishes.
In particular, $\hn$ is an Einstein connection in the sense of \cite{sw2}.
\end{thm}

Given a Lagrangian submanifold $M$ of an almost K\"ahler manifold, we also defined in \cite{sw2} a generalized mean curvature vector field $\widehat{H}$ in terms of the usual mean curvature vector
and the torsion $\hT$ of $\widehat{\nabla}$. In addition, we
proved that the restriction of  $i(\widehat{H})\omega$ to $M$ is a closed one form if  $\widehat{\nabla}$ is Einstein. Such a relation is known to be true on a Lagrangian
submanifold of a K\"ahler-Einstein manifold in which $\widehat{\nabla}$ is the Levi-Civita connection. This new characterization allows us to extend many known results regarding Lagrangian submanifolds of K\"ahler-Einstein manifolds to this more general setting. In particular, we found the cotangent bundle case to be analogous to the Calabi-Yau case in the following. Once we fix a Riemannian metric on the base, we can locally define the Lagrangian angle $\theta$ of a Lagrangian submanifold by taking the angle between the tangent space $T_pM$ and the tangent space of the fiber of $\pi: T^*\Sigma\rightarrow \Sigma $ (the vertical distribution) through any point $p\in M$. The generalized mean curvature vector $\hH$ is in fact dual to the form $d\theta$  which up to some constant is the Maslov
form with respect to  the canonical symplectic form on $T^*\Sigma$. There also exists a parallel $(n, 0)$-form $\Omega$  as a section of the canonical line bundle on any cotangent bundle. On a Lagrangian submanifold $M\subset N=T^*\Sigma$, we show that (Proposition \ref{Omega_angle}) the Lagrangian angle  is related to $\Omega$ by
\begin{equation}
e^{i\theta}=*(\Omega|_M)
\end{equation}
where $*$ is the Hodge star operator on $M$.

In \cite{sw2}, we also consider the generalized mean curvature flow with respect to $\hH$ (see a different generalized Lagrangian mean curvature flow studied by Behrndt \cite{Be}).  This is a family of moving submanifolds $M_t$, $t\in [0, T)$ such that the velocity vector at
each point is given by the generalized mean curvature vector of $M_t$ at that point. We prove that the parabolic flow is well-posed and preserves the Lagrangian condition \cite{sw2}.
The above interpretation of $\widehat{H}$ in terms of the Lagrangian angle indeed gives a heuristic reason why the latter holds on the linear level.
Therefore, the flow gives a canonical Lagrangian deformation in cotangent bundles.

When $\Sigma$ is compact and orientable, a conjecture that is often attributed to Arnol'd \cite{Arnold, Gromov, Lalonde-Sikorav} asks
if a compact, exact, orientable, embedded Lagranian $M$ in $T^*\Sigma$ can be deformed through exact Lagrangians to the zero section. We refer to \cite{fss} for the current development towards the conjecture from the perspective of symplectic topology. In relation to this question, we prove in this paper:

\begin{thm}\label{zero_Maslov} Suppose that $\Sigma$ is a compact Riemannian manifold.
 Suppose $M_t$, $t\in [0, T)$ is a smooth generalized Lagrangian mean curvature flow of compact Lagrangians in $T^*\Sigma$, if $M_0$ is exact and of vanishing Maslov class, so is $M_t$ for any $t\in [0, T)$.
 \end{thm}

This is proved by computing the evolution equation of the Lagrangian angle and the Liouville form along the flow.
 That the connection $\widehat{\nabla}$ is metric, complex and preserves the horizontal and vertical distributions is critical in studying the geometry of this flow.

The flow thus presents a natural candidate for the deformation of Lagrangian submanifolds in cotangent bundles. However, it is  known that there are many analytic difficulties even in the original Lagrangian mean curvature flow case, see  \cite{ne2, ne3}. As a first step towards understanding this flow, we focus on the graphical case in this article, i.e. when $M_t$ is defined by $du(x, t)$ for local potentials $u(x, t)$ defined on $(\Sigma, \sigma)$. In particular, we show in Proposition \ref{GMCF-u}
that the flow is equivalent to the following fully non-linear parabolic equation for $u$ ({\it the special Lagrangian evolution equation})
\begin{equation}\label{eq_u}\frac{\partial u}{\partial t}= \frac{1}{\sqrt{-1}}\ln \frac{\det( \sigma_{ij}+\sqrt{-1} u_{;ij} )}{\sqrt{\det \sigma_{ij}}\sqrt{\det(\sigma_{ij}+u_{;ik}  \sigma^{kl} u_{;lj} )}}\end{equation} where $u_{;ij}$ is  the Hessian of $u(x, t)$ with respect to the fixed metric $\sigma_{ij}$.
We prove the following stability theorem of the zero section when the base manifold is a standard round sphere.

\begin{thm}
When $(\Sigma, \sigma)$ is a standard round sphere of constant sectional curvature, the zero section in $T^*\Sigma$ is stable under the generalized Lagrangian mean curvature flow.
\end{thm}

Theorem 3 holds when the standard round sphere is replaced by a compact Riemannian manifold of positive sectional curvature. For the detailed statement and precise condition, see section \S 7. In particular, we show that the generalized Lagrangian mean curvature flow of any Lagrangian submanifold with small local potential in $C^2$ norm (the smallness can be effectively estimated) exists for all time and converges
to the zero section at infinity.
The case when the base metric is flat i.e. $\sigma_{ij}=\delta_{ij}$  is studied in \cite{sw1, Zhang, cch, ccy}.
In these cases, one can use the unitary group action to convert the condition of small $C^2$ norm into a convexity condition
(see section 4 in \cite{sw1} for this transformation). The convexity condition implies the standard $C^{2,\alpha}$ estimate of Krylov \cite{Krylov}
is applicable. In our case the base manifold is no longer flat and no such transformation exists, and we need to deal with the $C^{2, \alpha}$ estimate directly. A similar flow for holomorphic line bundles was considered in \cite{Jacob-Yau}.

The article is organized as follows. In \S 2, the almost K\"ahler geometry of cotangent bundles is reviewed and Theorem \ref{theo 1} is proved.
In \S 3, we review the geometry of Lagrangian submanifolds in the cotangent bundle, in particular we recall the generalized mean curvature vector. There we derive the relation between the Lagrangian angle and the parallel $(n, 0)$-form.
In  \S 4, we study the evolution equations under the generalized mean curvature flow in the cotangent bundle and prove Theorem \ref{zero_Maslov}.
In \S 5, we investigate the graphical case in which the Lagrangian submanifold is given by the graph of a closed one-form on $\Sigma$. In \S 6, we compute the evolution equations of different geometric quantities that will be used in the proof of the stability theorem.
In \S 7, we prove the
stability theorem Theorem 3.
Readers who are more interested in the PDE aspect of \eqref{eq_u} can move directly to \S 5.

The third named author is grateful to Pengfei Guan for raising the question about how to generalize the special Lagrangian equation to a Riemannian manifold.
He would like to thank  Tristan Collins, Adam Jacob, Conan Leung, and Xiangwen Zhang  for their interests and helpful discussions.
\section{Review of the geometry of cotangent bundles of Riemannian manifolds}

\subsection{The almost K\"ahler structure $\bf (\omega, J,g)$  on $\bf T^*\Sigma$}

We first review the geometry of cotangent bundles, some of which can be found in \cite{va} or \cite{yi}.

Let $(\Sigma,\sigma)$ be an $n$-dimensional Riemannian manifold with Riemannian metric $\sigma$. Let $\{q^j\}_{j=1\cdots n}$ be a local coordinate system on $\Sigma$.  Let $D$ be the covariant derivative (connection) and $\Lambda_{ij}^k$ be the Christoffel  symbols of $\sigma_{ij}$ with
\[D_{\frac{\partial}{\partial q^i}}\frac{\partial}{\partial q^j}=\Lambda_{ij}^k\frac{\partial}{\partial q^k}.\]

Let $\uddd Cijkl$ be the curvature tensor of $\sigma_{ij}$ with
\[C\left(\frac{\partial}{\partial q^k}, \frac{\partial}{\partial q^l}\right)\frac{\partial}{\partial
q^j}=D_{\frac{\partial}{\partial q^k}}D_{\frac{\partial}{\partial q^l}}\frac{\partial}{\partial q^j}-D_{\frac{\partial}{\partial q^l}}D_{\frac{\partial}{\partial q^k}}\frac{\partial}{\partial q^j}=\uddd Cijkl\frac{\partial}{\partial q^i}.\]
$\uddd Cijkl$ can be expressed by the Christoffel symbols by
\begin{equation}\label{curvature}\uddd Cijkl=\frac{\partial}{\partial q^k} \Lambda^i_{jl}-\frac{\partial}{\partial q^l}\Lambda^i_{jk}+\Lambda^i_{pk} \Lambda^p_{jl}-\Lambda^i_{pl}\Lambda^p_{jk}.
\end{equation}

Let $N:=T^*\Sigma$ be the cotangent bundle of $\Sigma$. We take the local coordinates  $\{q^i, p_i\}_{i=1\cdots n}$ on $T^*\Sigma$ such that on overlapping charts with coordinates $q^i, p_i$ and $\tilde{q}^i, \tilde{p}_i$, the transformation rule
\[\tilde{p}_i=\frac{\partial q^j}{\partial \tilde{q}^i}p_j\] holds.
Denote the Liouville form by $\lambda=p_i dq^i$ so that the canonical symplectic form by $\omega=\sum_{i=1}^n dq^i\wedge dp_i$
is given by
\begin{equation}\label{dlambda}
\omega=-d\lambda\,.
\end{equation}

Recall that $\{dq^i, \theta_i\}_{ i=1\dots n}$ form a basis for $T^*(T^*\Sigma)$ where
\begin{equation}\label{theta_i}
\theta_i=dp_i-\Lambda_{ih}^k p_k dq^h\,,\quad i=1,\dots, n
\end{equation}
 that is dual to the basis $\{X_i, \frac{\partial}{\partial p_i}\}_{i=1\cdots n}$ for $T(T^*\Sigma)$ where
\begin{equation}\label{X_i}
X_i=\frac{\partial}{\partial q^i}+\Lambda_{ih}^k p_k \frac{\partial}{\partial p_h}\,,\quad i=1,\dots,n.
\end{equation}

Denote
\[X^i=\sigma^{ik} X_k\\,\quad i=1,\dots  ,n.\]
The bundle projection $\pi:T^*\Sigma\to\Sigma$ then satisfies
$$d\pi(X_i)=\frac{\partial}{\partial q^i}\,,\quad d\pi\left(\frac{\partial}{\partial p_i}\right)=0\,.$$
Thus the connection $D$ generates two distributions $\mathscr{H}$, $\mathscr{V}$ in $T(T^*\Sigma)$, called the horizontal
and vertical distributions.
We summarize the properties in the following:

\begin{pro}\label{metric}
Let $N:=T^*\Sigma$ for a Riemannian manifold $(\Sigma, \sigma)$.  The horizontal distribution $\mathscr{H}\subset TN$ is spanned by $X^i$ and the vertical distribution
$\mathscr{V}$ by $\frac{\partial}{\partial p_i}$.
In terms of these bases, the Riemannian metric $g=\langle \cdot, \cdot \rangle$ on
$N$ (or on the tangent bundle $TN$ of $N$) satisfies
\[\left\langle \frac{\partial}{\partial p_i}, \frac{\partial}{\partial p_j}\right\rangle=\sigma^{ij},\quad\left\langle X^i, \frac{\partial}{\partial p_j}\right\rangle =0, \quad\text{and}\quad
\langle X^i, X^j\rangle=\sigma^{ij}.\]

In terms of $\theta_i$ and $dq^i$, this metric is
\[g(\cdot,\cdot)=\langle\cdot, \cdot \rangle=\sigma^{ij}\theta_i \otimes\theta_j+\sigma_{ij} dq^i \otimes dq^j.\]

The almost complex structure $J$ on $TN$ is defined by \[\omega(\cdot, \cdot)=\langle J\cdot, \cdot\rangle \] and it satisfies
\begin{equation}\label{J}
JX^i=\frac{\partial}{\partial p_i},\quad  J \frac{\partial}{\partial p_i}=-X^i, \quad\text{and}\quad Jdq^i=-\sigma^{ij} \theta_j.
\end{equation}

\end{pro}

$g$ is the Sasaki metric \cite{Sasaki} on the cotangent bundle $N=T^*\Sigma$.

\subsection{The connection, the curvature, and the torsion}
Now we recall the  connection $\hn$ (see \cite{va}) on $T(T^*\Sigma)$ that is compatible with the Riemannian metric $\langle\cdot, \cdot\rangle$ and the almost complex structure $J$ (i.e. the covariant derivative $\hn$ commutes with $J$). $\hn$ is defined by
\begin{equation}\label{connection}
\hn X^i=-\Lambda^i_{jk} dq^j \otimes X^k\,\,\text{and}\,\, \hn \frac{\partial}{\partial  p_i}=-\Lambda^i_{jk} dq^j \otimes \frac{\partial}{\partial p_k}.
\end{equation}

From these, we can compute the covariant derivative of any vector field. For example, by \eqref{X_i}, we have
\begin{equation}\label{connection2}\widehat{\nabla}_{\frac{\partial}{\partial p_j}}\frac{\partial}{\partial q^i}=- \Lambda^j_{ik}\frac{\partial}{\partial p_k}.\end{equation}

We notice that this connection preserve the horizontal and the vertical distribution. Also $X^i$ and $\frac{\partial}{\partial p_i}$ are parallel in the fiber direction.

Let $\hR$ be the curvature tensor of $\hn$. Since $\hn$ is complex and metric, the Ricci form $\hrho$ is given by
\begin{equation}\label{Ricci}\hrho(V,W):=\frac{1}{2}\sum_{\alpha=1}^{2n}g(\hR(V,W)Je_\alpha,e_\alpha)
=\frac{1}{2}\sum_{\alpha=1}^{2n}\omega(\hR(V,W)e_\alpha,e_\alpha)\,,\end{equation}
where $e_\alpha$ is an arbitrary orthonormal basis of $TN$.

We recall the definition of an Einstein connection from \cite{sw2}:
\begin{dfn}\label{einstein connection}
A metric and complex connection $\hn$ on an almost K\"ahler manifold $(N,\omega,J,g)$
is called {\sl Einstein}, if the Ricci form of $\hn$ satisfies
$$\hrho=f\omega$$
for some smooth function $f$ on $N$.
\end{dfn}

We denote the projection of $TN$ onto the horizontal distribution $\mathscr{H}$ by $\pi_1$ and the projection onto the vertical distribution
$\mathscr{V}$  by $\pi_2$. In terms of $dq^i$ and $\theta_i$, we have
\[\pi_1= dq^i\otimes X_i \quad\text{and}\quad \pi_2=\theta_i\otimes\frac{\partial}{\partial p_i}.\]

Since $J$ interchanges $\mathscr{H}$ and $\mathscr{V}$ we get
\begin{equation}\label{eq cot commute}
J\pi_1=\pi_2J\,,\quad J\pi_2=\pi_1 J\,.
\end{equation}

With respect to these structures, we define:
\begin{dfn}
The $n$-form $\Omega$ is defined as
\begin{equation}\label{Omega_0}
\Omega=\sqrt{\det \sigma_{ij}} (dq^1-\sqrt{-1}J dq^1)\wedge\cdots\wedge (dq^n-\sqrt{-1} Jdq^n).
\end{equation}
$\Omega$ can be viewed as an $(n, 0)$-form in the sense that
\begin{equation}\label{hol}
\Omega(JV_1,V_2,\dots, V_n)=\sqrt{-1} \Omega(V_1, \dots, V_n).
\end{equation}
\end{dfn}

\begin{pro} The $(n, 0)$ form $\Omega$  on $N=T^*\Sigma$ is parallel with respect  to the connection  $\widehat{\nabla}$.
\end{pro}

\begin{proof}
We begin by computing $\hn dq^i$. Consider
\begin{eqnarray}
(\hn dq^i)(X^k)
&=&d[dq^i(X^k)]-dq^i(\widehat{\nabla} X^k)\nonumber\\
&=&d( \sigma^{ki})+\Lambda_{pq}^k \sigma^{pi}dq^q\nonumber\\
&=&-\sigma^{km}\Lambda_{ms}^idq^s.
\end{eqnarray}

On the other hand, $(\hn dq^i)\left(\frac{\partial}{\partial p_k}\right)=0$. Therefore,
\begin{equation}\label{nabla dxi}
\hn dq^i=-\Lambda_{ms}^i dq^m \otimes dq^s.
\end{equation}

The proposition follows by putting the together the above formula and the following standard calculation
\[d\sqrt{\det \sigma_{ij}}=\Lambda_{sk}^k \sqrt{\det \sigma_{ij}}dq^s.\]
\end{proof}

\subsection{Proof of Theorem \ref{theo 1}}

By equation \eqref{connection}, the curvature $\widehat{R}$ of $\hn$ is computed as
\begin{eqnarray}
\widehat{R}\left(\frac{\partial}{\partial q^k}, \frac{\partial}{\partial q^l}\right) X^i
&=&
\hn_{\frac{\partial}{\partial q^k}} \hn_\frac{\partial}{\partial q^l} X^i-\hn_{\frac{\partial}{\partial q^l}} \hn_\frac{\partial}{\partial q^k} X^i\nonumber\\
&=&-\left(\frac{\partial}{\partial q^k} \Lambda^i_{jl}-\frac{\partial}{\partial q^l}\Lambda^i_{jk}-\Lambda^i_{pl}\Lambda^p_{jk}
+\Lambda^i_{pk} \Lambda^p_{jl} \right)X^j.\nonumber
\end{eqnarray}
Likewise,
\[\widehat{R}\left(\frac{\partial}{\partial q^k}, \frac{\partial}{\partial q^l}\right) \frac{\partial}{\partial p_i}=-\left(\frac{\partial}{\partial q^k} \Lambda^i_{jl}-\frac{\partial}{\partial q^l}\Lambda^i_{jk}-\Lambda^i_{pl}\Lambda^p_{jk}
+\Lambda^i_{pk} \Lambda^p_{jl} \right)\frac{\partial}{\partial p_j}.\]

Therefore, we have
\begin{equation}
\widehat{R}\left(\frac{\partial}{\partial q^k}, \frac{\partial}{\partial q^l}\right) X^i=-\uddd Cijkl X^j
\text{ and }
\widehat{R}\left(\frac{\partial}{\partial q^k}, \frac{\partial}{\partial q^l}\right) \frac{\partial}{\partial p_i}=-\uddd Cijkl\frac{\partial}{\partial p_j}.
\end{equation}

In view of these relations, the Ricci form $\hrho$, see \eqref{Ricci}, vanishes since $J$ is an isomorphism between the vertical and horizontal
distributions.

\section{The generalized mean curvature and the Lagrangian angle in cotangent bundles}
In the last sections we have seen that the cotangent bundle $N=T^*\Sigma$ of a Riemannian
manifold admits a naturally defined almost K\"ahler structure $(\omega,J,g)$ and a
canonical connection $\hn$ that is metric, symplectic and  has torsion $\hT$,
essentially given by the curvature of the underlying base manifold $(\Sigma,\sigma)$. Moreover the
Ricci form $\widehat\rho$ of $\hn$ vanishes. From now on we will assume that $(N,\omega,J,g,\hn)$ is
such a cotangent bundle.

We now recall the definition of the generalized mean curvature vector field of a Lagrangian immersion $F:M\to N$ and relate it to the Lagrangian angle through the holomorphic $n$-form $\Omega$ introduced in the previous section. We shall identify $M$ with the image of the Lagrangian immersion and refer $M$ as a Lagrangian submanifold when there is no confusion. Let $e_i, i=1\cdots n$ be an orthonormal basis with respect to the induced metric on $M$ by the immersion $F$. We recall the generalized mean curvature form on $M$ is
\begin{equation}\label{mcform}\mu_i=\sum_k \langle \hn_{e_i} e_k , J e_k\rangle\end{equation} and the generalized mean curvature vector $\hH$ is
\begin{equation}\label{mcvector}\hH=\sum_i \mu_iJe_i\,.\end{equation}

Consequently, the generalized mean curvature vector is dual to the mean curvature form in the sense that
\begin{equation}\label{eq maslov}
i(\hH)\omega|_M=-\mu.
\end{equation}

We recall that the Lagrangian angle of a Lagrangian subspace $L_1$ in $\C^n$ with respect to another fixed Lagrangian subspace $L_0$ is given by the argument of $\det U$ where $U$ is a unitary $n\times n$ matrix such that $L_1=UL_0$. Effectively,  we choose an orthonormal basis $e^a_1,\dots, e^a_n$ for $L_a$, $a\in\{0,1\}$, and set
$$\epsilon_i^a=\frac{1}{\sqrt{2}}(e^a_i-\sqrt{-1} Je^a_i)$$
to be the associated holomorphic basis. If $\epsilon_i^1=\gamma_i^j \epsilon_j^0$, then $\det \gamma_i^j$ is the Lagrangian angle of $L_1$ with respect to $L_0$. We derive a formula for the Lagrangian angle in terms of arbitrary bases.
\begin{lem}\label{angle}
Suppose $(V, \langle\cdot, \cdot\rangle)$ is a $2n$-dimensional (real) inner product space with a compatible almost complex structure $J$ (i.e.
$J$ is an isometry and $J^2=-I$). Let $L_0$ be a fixed Lagrangian subspace of $V$ spanned by $\bar{v}_1, \dots ,\bar{v}_n$. Suppose $L_1$ is another Lagrangian subspace spanned by $v_1, \dots, v_n$. Suppose $v_i=\sum_{j=1}^n \alpha_{ij}\bar{v}_j+\sum_{j=1}^n\beta_{ij} J\bar{v}_j$ for $i=1,\dots ,n$. Then the Lagrangian angle $\theta$ of $L_1$ with respect to $L_0$ is the argument of $\det (\alpha_{ij}+\sqrt{-1} \beta_{ij})$. In fact, they are related by
\begin{equation}\label{eq theta}\frac{\det (\alpha_{ij}+\sqrt{-1} \beta_{ij}) \sqrt{\det\langle \bar{v}_i, \bar{v}_j\rangle}}{\sqrt{\det\langle v_i, v_j\rangle}}=e^{\sqrt{-1}\theta}.\end{equation} \end{lem}
\begin{proof} Direct calculation.\end{proof}

Note that by this formula the Lagrangian angle is not uniquely defined but it is defined up to adding
an integer multiple of $2\pi$.

Suppose $F: M\rightarrow N=T^*\Sigma$ is a Lagrangian immersion.
We consider the Lagrangian angles with respect to the horizontal distribution $\mathscr{H}$ and the vertical distribution $\mathscr{V}$, which differ by a constant. For our purpose, we shall use $\theta$ to denote the Lagrangian angle with respect to the horizontal distribution.

\begin{pro} Suppose a Lagrangian submanifold  of $N=T^*\Sigma$ is given by $F:M\rightarrow N$. Let $\{F_i\}_{i=1\cdots n}$ be an arbitrary basis
tangential to $M$.
Then the Lagrangian angle $\theta$ with respect to the horizontal distribution is
\begin{eqnarray}
\sqrt{-1}\theta&=&\ln \det \left(\langle F_i, X^j\rangle +\sqrt{-1} \langle F_i, \frac{\partial}{\partial p_j} \rangle\right)\nonumber\\
&&+\frac{1}{2}\ln\det \sigma_{ij}-\frac{1}{2}\ln \det G_{ij},\nonumber
\end{eqnarray} where $G_{ij}=\langle F_i, F_j\rangle$.
\end{pro}

\begin{proof}

Each $F_i$ can be expressed in terms of $X^j$ and $p_j$,

\begin{equation}\label{eq F-i}F_i=\left\langle F_i, X^k\right\rangle \sigma_{kj}X^j
+\left\langle F_i, \frac{\partial}{\partial p_k} \right\rangle \sigma_{kj}\frac{\partial}{\partial p_j}.
\end{equation}

Since $JX^i=\frac{\partial}{\partial p_i}$, by Lemma \ref{angle}, the Lagrangian angle $\theta$ with respect to the horizontal
distribution spanned by $\{X^i\}_{i=1\cdots n}$ is the argument of
\[\det \left(\langle F_i, X^k\rangle +\sqrt{-1} \langle F_i, \frac{\partial}{\partial p_k} \rangle\right). \]
Using $v_i=F_i$, $\overline{v}_i=\sigma_{il}X^l$, $\alpha_{ij}=\langle F_i, X^j\rangle$
, $\beta_{ij}=\left\langle F_i, \frac{\partial}{\partial p_j} \right\rangle$, $\langle \bar{v}_i, \bar{v}_j\rangle=\sigma_{ij}$, we obtain
the formula from \eqref{eq theta}.

\end{proof}

On the other hand, the Lagrangian angle with respect to the vertical distribution spanned by $\{\frac{\partial}{\partial p_j}\}_{j=1\cdots}$ is
the argument of
\[\det \left(    \langle F_i, \frac{\partial}{\partial p_k} \rangle-\sqrt{-1}  \langle F_i, X^k\rangle \right). \] Therefore, the two Lagrangian
angles differ by a multiply of $\frac{\pi}{2}$.

Another way to compute the Lagrangian angle with respect to the horizontal
distribution is to consider the restriction of the $(n, 0)$
form $\Omega$ to $M$.

\begin{pro} \label{Omega_angle}Suppose $\Omega$ is the $n$-form given by \eqref{Omega_0}, then for a Lagrangian immersion $F:M\to T^*\Sigma$,
\[*(\Omega|_M)=e^{\sqrt{-1} \theta},\] where $*$ is the Hodge star on $M$.
\end{pro}

\begin{proof}
Given any basis $\{F_i\}_{i=1,\cdots n}$ tangential to $M$,
$*(\Omega|_M)=\frac{\Omega (F_1, \dots, F_n)}{\sqrt{\det\dd Gij}}$ where $G_{ij}=\langle F_i, F_j\rangle$.  We calculate
\[(dq^k-\sqrt{-1}J dq^k)(F_i)=\langle F_i, X^k\rangle+\sqrt{-1}\left\langle F_i, \frac{\partial}{\partial p_k}\right\rangle\] where equation (\ref{J}) is used.
\end{proof}

\begin{pro}\label{eq_H-theta}For a Lagrangian immersion $F:M\to T^*\Sigma$,
the generalized mean curvature vector and the Lagrangian angle are related by \[\widehat{H}=J\nabla \theta\] where $\nabla$ is the gradient operator on $M$ with respect to the induced metric on $M$.
\end{pro}

\begin{proof} In view of  \eqref{mcform}, \eqref{mcvector}, and Proposition \ref{Omega_angle}, it suffices to prove that \begin{equation}\label{dlogOmega} d\ln (*\Omega)= \sqrt{-1} \mu,\end{equation} where $*\Omega=*(\Omega|_M)$.
Let $e_1,\dots,e_n$ be an orthonormal frame tangential to $M$. Using the fact that $\Omega$ is parallel with respect to $\widehat{\nabla}$, we compute
\[e_i(*\Omega)=\Omega(\widehat{\nabla}_{e_i} e_1, e_2, \dots, e_n)+\cdots +\Omega(e_1,\dots, e_{n-1}, \widehat{\nabla}_{e_i} e_n).\] Since the tangential part of $\widehat{\nabla}_{e_i} e_1$ only involves $e_2, \dots, e_n$, the first term becomes $\Omega((\widehat{\nabla}_{e_i} e_1)^\perp, e_2, \dots, e_n)$. Likewise for other terms. On the other hand, we have $(\widehat{\nabla}_{e_i} e_k)^\perp=\langle \widehat{\nabla}_{e_i} e_k, Je_l\rangle Je_l$.
Using the property that $\Omega$ is a holomorphic $n$-form, see equation (\ref{hol}), we derive
\[\Omega((\widehat{\nabla}_{e_i} e_1)^\perp, e_2, \dots, e_n)=\sqrt{-1}\langle\widehat{\nabla}_{e_i} e_1, Je_1\rangle *\Omega.\] Summing up from $i=1,\dots, n$, we arrive at the desired formula.
\end{proof}

\section{The generalized mean curvature flow in cotangent bundles}

We  derive evolution equations along the generalized mean curvature flow in cotangent bundles for the Lagrangian angle
and the Liouville form.

Before that, let us recall some facts about the torsion connection from \cite{sw2}. In Lemma 2 in \cite{sw2}, it is shown that the torsion connection $\widehat{\nabla}$ and the Levi-Civita connection $\widetilde{\nabla}$ on $N$ are related by
\begin{equation}\label{connection_rel}2\langle \hn_X Y-\widetilde{\nabla}_X Y, Z\rangle
=\langle \hT(X, Y), Z\rangle+\langle \hT(Z, X), Y\rangle
+\langle \hT(Z, Y), X\rangle.\end{equation}

In particular, for a tangent vector field $X$ on a Lagrangian submanifold $M$, we have
\begin{equation}\label{eq divergence}
\sum_{k=1}^n\langle\hn_{e_k} X, e_k\rangle=div_M X+\sum_{i=1}^n \langle \hT(e_k, X), e_k\rangle,
\end{equation}
where $\{e_k\}_{k=1,\cdots, n}$ is an orthonormal basis of $TM$.

We recall that a smooth family of
Lagrangian immersions
\[F:M \times [0, T)\rightarrow N=T^*\Sigma\] satisfies the generalized mean curvature flow, if
\begin{equation}
\frac{\partial F}{\partial t}(x, t)=\hH(x, t)\,,\quad\text{and}\quad F(M, 0)=M_0
\end{equation}
where $\hH(x,t)$ is the generalized mean curvature vector of the almost Lagrangian submanifold
$M_t=F(M, t)$ at $F(x,t)$. It was proved in \cite{sw2} that the generalized  mean curvature flow preserves the Lagrangian condition.
In the following calculations, we fix a local coordinate system $(x^1, \cdots, x^n)$ on the domain $M$ and consider $F_i=\frac{\partial F}{\partial x^i}=dF(\frac{\partial}{\partial x^i}), i=1, \cdots, n$ a tangential basis on the moving submanifolds $M_t$.

\begin{lem} Along the generalized mean curvature flow $M_t$ in the cotangent bundle of a Riemannian
manifold,  the Lagrangian angle $\theta$ satisfies
\begin{equation}\label{lag_ang}\frac{\partial}{\partial t} \theta
=\Delta \theta+\sum_{k=1}^n \left(\langle \hT(J\hH,e_k), e_k\rangle
-\langle J\widehat{T}(\hH, e_k), e_k\rangle\right),\end{equation} for any orthonormal basis $\{e_k\}_{k=1\cdots n}$ on $M_t$.
\end{lem}

\begin{proof}
We compute $\frac{\partial}{\partial t}(*\Omega)$ where $*\Omega=\frac{\Omega(F_1, \dots, F_n)}{\sqrt{\det G_{ij}}}$ and $G_{ij}=\langle F_i, F_j\rangle$ is the induced metric on the Lagrangian submanifold $M_t$:
\[\frac{\partial}{\partial t}(*\Omega)=\frac{1}{\sqrt{\det G_{ij}}}\frac{\partial}{\partial t}(\Omega (F_1, \dots, F_n))-*\Omega \frac{\partial}{\partial t}\ln \sqrt{\det G_{ij}}.\]
Since $\Omega$ is parallel with respect to $\widehat{\nabla}$, we derive
\[\frac{\partial}{\partial t} \Omega(F_1, \dots, F_n)=\Omega (\widehat{\nabla}_{\hH}F_1, F_2, \dots, F_n)+\cdots +\Omega (F_1, \dots, F_{n-1}, \widehat{\nabla}_{\hH}  F_n ).\] Decomposing $\widehat{\nabla}_{\hH}F_i=(\widehat{\nabla}_{\hH}F_i)^\perp+(\widehat{\nabla}_{\hH}F_i)^\top$, and noting that $(\widehat{\nabla}_{\hH}F_i)^\top=\langle \widehat{\nabla}_{\hH}F_i, F_j\rangle G^{jk} F_k$, we derive that $\frac{\partial}{\partial t} \Omega(F_1, \dots, F_n)$ is equal to
\[\begin{split}&\langle \widehat{\nabla}_{\hH} {F_i}, F_j\rangle G^{ij} \Omega(F_1, \dots, F_n)+\Omega \left((\widehat{\nabla}_{\hH} {F_1} )^\perp, F_2, \dots, F_n\right)\\
+ &\cdots+
\Omega \left(F_1, F_2, \dots, (\widehat{\nabla}_{\hH} {F_n} )^\perp\right).\end{split}\]

On the other hand,
\[\frac{\partial}{\partial t}\ln \sqrt{\det G_{ij}}=\langle \widehat{\nabla}_{\hH} {F_i} , F_j\rangle (G^{-1})^{ij}.\]
Therefore,
\[\begin{split}\frac{\partial}{\partial t}*\Omega
&=\frac{1}{\sqrt{\det G_{ij}}}\Big[\Omega ([\widehat{\nabla}_{F_1} \hH+\widehat{T}(\hH, F_1)]^\perp, F_2, \dots, F_n)+\cdots\\
&+\Omega ([F_1, F_2, \dots, \widehat{\nabla}_{F_n} \hH+\widehat{T}(\hH, F_n)]^\perp)\Big].\\
\end{split}\]
In the rest of the calculation we can choose coordinates $x^i$ at any point of interest so that $\{F_i=e_i\}_{i=1,\cdots, n}$ is orthonormal. We compute
\[(\widehat{\nabla}_{e_1}\hH)^\perp=\langle \widehat{\nabla}_{e_1} \hH, Je_k\rangle Je_k\] and thus
\[\Omega( \langle \widehat{\nabla}_{e_1} \hH, Je_k\rangle Je_k, e_2, \dots, e_n)=\sqrt{-1} \langle \widehat{\nabla}_{e_1} \hH, Je_1\rangle *\Omega.\]

We can likewise compute other terms and obtain
\[\begin{split}\frac{\partial}{\partial t}(*\Omega)=\sqrt{-1}\sum_{k=1}^n (\langle \widehat{\nabla}_{e_k} \hH, Je_k\rangle +\langle \widehat{T}(\hH, e_k), Je_k\rangle)*\Omega.
\end{split}\] or
\[\begin{split}\frac{\partial}{\partial t}\theta=\sum_{k=1}^n\left(\langle \widehat{\nabla}_{e_k} \hH, Je_k\rangle +\langle \widehat{T}(\hH, e_k), Je_k\rangle\right).
\end{split}\]

Now since $\hn J=0$ and $\hH=J\nabla\theta$ we have
\begin{eqnarray}
\frac{\partial}{\partial t}\theta
&=&\sum_{k=1}^n\left(\langle \widehat{\nabla}_{e_k} \hH, Je_k\rangle
+\langle \widehat{T}(\hH, e_k), Je_k\rangle\right)\nonumber\\
&=&\sum_{k=1}^n\left(\langle \widehat{\nabla}_{e_k} \nabla\theta, e_k\rangle
+\langle \widehat{T}(\hH, e_k), Je_k\rangle\right)\nonumber\\
&\overset{(\ref{eq divergence})}{=}&\Delta\theta
+\sum_{k=1}^n \left(\langle \hT(e_k, \nabla \theta), e_k\rangle
+\langle \widehat{T}(\hH, e_k), Je_k\rangle\right)\nonumber\\
&=&\Delta\theta
+\sum_{k=1}^n \left(\langle \hT(J\hH,e_k), e_k\rangle
-\langle J\widehat{T}(\hH, e_k), e_k\rangle\right)\,.\nonumber
\end{eqnarray}
\end{proof}

\begin{lem}
Along a generalized Lagrangian mean curvature flow, the Liouville form evolves as
\begin{equation}\label{lio}
\frac{\partial}{\partial t} F^*\lambda= d(\lambda(\hH))+\mu.\end{equation}
\end{lem}

\begin{proof}
This follows from the Cartan's formula that the Lie derivative is $L_X=di(X)+i(X)d$.
Note that by equations (\ref{dlambda}) and (\ref{eq maslov}) we have $i(\hH)d\lambda=-i(\hH)\omega=\mu$.
\end{proof}

We recall the statement of Theorem \ref{zero_Maslov} and prove it.
\vskip 10pt
 \noindent{\bf Theorem 2}
 \textit{ Suppose $M_t$, $t\in [0, T)$ is a smooth generalized Lagrangian mean curvature flow in $T^*\Sigma$, if $M_0$ is exact and of vanishing Maslov class, so is $M_t$ for any $t\in [0, T)$.}
\vskip 10pt

\begin{proof}
By differentiating both sides in (\ref{lag_ang}), we see $\mu$ always changes by some exact form. This shows vanishing Maslov class is preserved. Since the Maslov class of $M_t$ vanishes for all
$t$, the Lagrangian angle can be chosen to be a single value function $\theta$ for all $t$.
Now equation (\ref{lio}) can be rewritten as
\[\frac{\partial}{\partial t} F^*\lambda= d(\lambda(\hH) +\theta)\] and we see that being exact is also preserved.
\end{proof}

\section{The graphical case}
In this section we consider the generalized Lagrangian mean curvature flow of
Lagrangian graphs in the cotangent bundle $T^*\Sigma$ of a Riemannian manifold $(\Sigma,\sigma)$
that are induced by $1$-forms
on $\Sigma$. The graphical case is interesting from an analytic point of view and can be seen as a "test case"
for the more general non-graphical situation. Let $M\subset T^*\Sigma$ be the
graph of a smooth $1$-form $\eta\in\Omega^1(\Sigma)$
on $\Sigma$. In this case, we can use local coordinates $q^1,\dots, q^n$ on $\Sigma$ to parametrize $M$ and the graph of $\eta$ defined
by
$$F:\Sigma \to T^*\Sigma\,,\quad F(q)=(q,\eta(q))$$
is Lagrangian if and only if $\eta$ is closed (hence locally exact, i.e. $\eta=du$ for a
locally defined potential $u$ on $\Sigma$). In the sequel we will always assume that $\eta$ is closed.

We remark the the calculation in this section is non-parametric, as opposed to the parametric calculation in the last section.

The tangent space to the image of $F$ is spanned by the basis
$$F_i:=\frac{\partial F}{\partial q^i}=X_i+\eta_{j;i}\frac{\partial}{\partial p_j}\,,$$
where $\eta_{j;i}=\partial_i\eta_j-\Lambda_{ij}^k \eta_k$ denotes the covariant derivative of the one-form $\eta$ with respect to the fixed background metric $\sigma$ on $\Sigma$.

The Lagrangian angle can be computed in terms of $\eta_{j;i}$.

\begin{pro}
Suppose $M$ is a Lagrangian submanifold of $T^*\Sigma$ defined as the graph of a closed $1$-form
$\eta\in\Omega^1(\Sigma)$. Then the Lagrangian angle $\theta$
of $M$ with respect to
to the horizontal distribution is
\[e^{\sqrt{-1} \theta}=\frac{\det(\sigma_{ij}+\sqrt{-1} \eta_{j;i} )}{\sqrt{\det \sigma_{ij}}\sqrt{\det(G_{ij})}}\,,\]
where $\sigma_{ij}$ is the metric on $M$, $\eta_{j;i}$ is the
covariant derivative of $\eta$ with respect to $\sigma_{ij}$, and $G_{ij}=\sigma_{ij}+\sigma^{kl} \eta_{k;i} \eta_{l;j}$ is the induced metric on $M$.

\end{pro}
\begin{proof}
This follows from Lemma \ref{angle} with \[v_i=\frac{\partial F}{\partial q^i}= \sigma_{ij} X^j+\eta_{j;i} JX^j.     \]

\end{proof}

The generalized mean curvature flow of graphs can be expressed locally as a fully nonlinear parabolic equation for the locally defined potential function $u$ (with $du=\eta$) on $\Sigma$.

\begin{pro}\label{GMCF-u}
Suppose $M_t$, $t\in [0, T)$ is a generalized mean curvature flow such that each $M_t$ is locally given as the graph of a closed one-form $\eta_t$ with local potential $u(\cdot, t)$ on $\Sigma$. The flow is then up to a tangential diffeomorphism equivalent to
\begin{equation}\label{eq flow}
\frac{\partial u}{\partial t}=\theta=\frac{1}{\sqrt{-1}} \ln   \frac{\det(\sigma_{ij}+\sqrt{-1}  u_{;ij} )}{\sqrt{\det \sigma_{ij}}\sqrt{\det G_{ij}}}
\end{equation}
where $\sigma_{ij}$ is the metric on $\Sigma$, $u_{;ij}=\partial_i\partial_j u-\Lambda_{ij}^k\partial_k u$ is the Hessian of $u$ with respect to $\sigma_{ij}$, and $G_{ij}=\sigma_{ij}+u_{;ik} \,\, \sigma^{kl}  u_{;lj}$ is the induced metric on $M_t$.
\end{pro}

\begin{proof}
We parametrize the flow by
$$F(q, t)=\left(q, \frac{\partial u}{\partial q^1}(q,t), \dots,\frac{\partial u}{\partial q^n}(q,t)\right)\,,$$
thus  $\frac{\partial F}{\partial t}=\frac{\partial}{\partial t}(\frac{\partial u}{\partial q^i})\frac{\partial }{\partial p_i}$ and the mean curvature vector $\hH$ is computed from (\ref{eq maslov})
\[\frac{\partial \theta}{\partial q^i} (G^{-1})^{ij}(\sigma_{jk} \frac{\partial}{\partial p_k}-u_{;jk}X^k).\] We claim that the normal part $(\frac{\partial F}{\partial t})^\perp$ of $\frac{\partial F}{\partial t}$ is
\[\frac{\partial}{\partial t}(\frac{\partial u}{\partial q^i})(G^{-1})^{ij}(\sigma_{jk} \frac{\partial}{\partial p_k}- u_{;jk}X^k).\] Equating coefficients in $(\frac{\partial F}{\partial t})^\perp=\hH$ yields \[\frac{\partial}{\partial t}(\frac{\partial u}{\partial q^i})=\frac{\partial \theta}{\partial q^i}, i=1\cdots n.\]
The desired equation is obtained by integration. It suffices to show that the normal part of $\frac{\partial}{\partial p_i}$ is
\[(G^{-1})^{ij}(\sigma_{jk} \frac{\partial}{\partial p_k}-u_{;jk}X^k)=(G^{-1})^{ij} J(\frac{\partial F}{\partial q^j})\] which follows from the fact that
\[\frac{\partial}{\partial p_i}-(G^{-1})^{ij} J(\frac{\partial F}{\partial q^j})\] is tangential.
\end{proof}

\begin{rem}
We remark that if $M_t$ remains graphical, there are two ways to parametrize the flow. The first way is the parametric flow in which the velocity vector at each point is  the mean curvature vector and thus represents a normal motion.  We fix a domain manifold and pull back the induced metric as a time-dependent metric defined on the domain. In particular, the equations derived in  \S4 are all with respect to this parametrization. The second way is the so called ``non-parametric flow" in which the velocity vector is a vertical vector, in fact, the vertical component of the mean curvature vector.
In this case, we may take the domain manifold to be the base manifold with the fixed background metric.
In the first case, it is natural to pull back a geometric quantity to the domain and then use the (time-dependent) induced metric to measure it.
In the second case, we project the quantity to the base manifold and use the fixed background metric. All calculations in \S6 and \S7 are with
respect to the non-parametric flow.

\end{rem}

\section{Graphical Lagrangian mean curvature flow in the cotangent bundles of Riemannian manifolds}
\subsection{
The special Lagrangian evolution equation on a Riemannian manifold}

Let $(\Sigma, \sigma)$ be an $n$-dimensional Riemannian manifold with Riemannian metric $\sigma_{ij}$ in a local coordinate system. Given a smooth function $u$ on $\Sigma$, let $u_{;ij} $ be the Hessian of $u$ with respect to the base metric $\sigma_{ij}$. Similarly, $u_{;ijk}$,  $u_{;ijkl}$, etc., denote higher order covariant derivatives of $u$.
From the definition of curvature \eqref{curvature}, we recall the following commutation formulae:
\begin{equation} \label{commutation}
\begin{split} u_{;pqk}-u_{;pkq}&=u_{;l} C^l_{\,\,\,pqk}\\
u_{;kpqi}-u_{;kpiq}&=u_{;lp} C^l_{\,\,\,kqi}+u_{;kl} C^l_{\,\,\,pqi}\\
u_{;mkpqi}-u_{;mkpiq}&=u_{;lkp} C^l_{\,\,\,mqi}
+u_{;mlp} C^l_{\,\,\,kqi}
+u_{;mkl} C^l_{\,\,\,pqi}.
\end{split}\end{equation}

$du$, as a closed one-form, defines a Lagrangian submanifold of the cotangent bundle of $\Sigma$. The Lagrangian angle (with respect to the horizontal distribution) of the graph of $du$ is defined as
\begin{equation}\label{theta}\theta= \frac{1}{\sqrt{-1}}\ln \frac{\det( \sigma_{ij}+\sqrt{-1} u_{;ij} )}{\sqrt{\det \sigma_{ij}}\sqrt{\det(\sigma_{ij}+u_{;ik}  \sigma^{kl} u_{;lj} )}}.\end{equation}

The generalized Lagrangian mean curvature flow defined in the previous section corresponds to the following nonlinear evolution equation of $u$.

\begin{dfn} Let $(\Sigma, \sigma)$ be a Riemannian manifold, a smooth function $u(q, t)$ defined on $\Sigma \times [0, T)$ is said to satisfy  the {\it special Lagrangian evolution equation} if
\begin{equation}\label{eq_u-1}\frac{\partial u}{\partial t}(q, t)=\theta (q, t)= \frac{1}{\sqrt{-1}}\ln \frac{\det( \sigma_{ij}+\sqrt{-1} u_{;ij} )}{\sqrt{\det \sigma_{ij}}\sqrt{\det(\sigma_{ij}+u_{;ik}  \sigma^{kl} u_{;lj} )}}\end{equation} where $u_{;ij}$ is  the Hessian of $u(q, t)$ with respect to the fixed metric $\sigma_{ij}$.
\end{dfn}

Let \begin{equation}\label{eq_G} G_{ij}=\sigma_{ij}+u_{;ik} \sigma^{kl} u_{;lj}\end{equation} be the $(0, 2)$ tensor on $\Sigma$ and $(G^{-1})^{ij}$ be the $(2,0)$ tensor on $\Sigma$ such that $G_{ij} (G^{-1})^{jk}=\delta_i^k$.
The following calculation is on the base manifold $\Sigma$ and indexes of tensors are raised or lowered by the base metric $\sigma$ which is time-independent. All derivatives are covariant derivatives with respect to $\sigma_{ij}$.

\begin{lem} \label{theta_derivative}
The derivative of $\theta$ is given by
\begin{equation}\label{eq_dtheta}\theta_{;k}=(G^{-1})^{ij}u_{;ijk}.\end{equation}
\end{lem}

\begin{proof}
Define $\gamma_{ij}=\sigma_{ij}+\sqrt{-1}{u_{;ij}}$, we compute $\gamma_{ij} \sigma^{jk} (\sigma_{kl}-\sqrt{-1}{u_{;kl}})=G_{il}$. Thus the inverse of $\gamma_{ij}$ is  $\sigma^{jm}(\sigma_{ml}-\sqrt{-1}{u_{;ml}})(G^{-1})^{lp}$. Therefore
\[\begin{split}&\sqrt{-1}\theta_{;k}\\
=&(\gamma_{ij})_{;k} \sigma^{im}(\sigma_{ml}-\sqrt{-1}{u_{;ml}})(G^{-1})^{lj}-\frac{1}{2} G_{ij;k} (G^{-1})^{ij}\\
=&\sqrt{-1} (G^{-1})^{ij} u_{;ijk}. \end{split}\]
\end{proof}

Now suppose $M_t$, $t\in [0, T)$ is a generalized Lagrangian mean curvature flow such that each $M_t$ is given as the graph of a closed one-form $\eta=du$.

Taking the derivative of \eqref{eq_u-1},  in view of Lemma \ref{theta_derivative}, we obtain
\begin{equation}\label{eq_du}\frac{\partial}{\partial t}u_{;k} =(G^{-1})^{ij}u_{;ijk},\end{equation}  which is equivalent to the generalized Lagrangian mean curvature flow by Proposition \ref{GMCF-u}.

We first derive the evolution of  the length square of $du$ with respect to the metric $\sigma$.
\begin{lem} Suppose $u$ is a solution the evolution equation \eqref{eq_u-1} on a Riemannian manifold $(\Sigma, \sigma)$, then $\vartheta=\sigma^{ij}u_{;i}u_{;j}$ satisfies the following evolution equation:
\begin{equation}\label{eq_eta2}
\frac{\partial }{\partial t}\vartheta-(G^{-1})^{ij}\vartheta_{;ij}
=-2 \sigma^{ij}(G^{-1})^{pq}u_{;ip}u_{;jq}+ 2\sigma^{ij}(G^{-1})^{pq}C^l_{pqi}u_{;l}u_{;j}.
\end{equation}
\end{lem}

\begin{proof}
A straightforward calculation using \eqref{eq_du} yields $$\frac{\partial}{\partial t} \vartheta=2\sigma^{ij}(G^{-1})^{pq}u_{;pqi}u_{;j}$$ and
$$(G^{-1})^{ij}\vartheta_{;ij}=2\sigma^{ij}(G^{-1})^{pq}(u_{;ip}u_{;jq}+u_{;ipq}u_{;j}).$$
The desired equation follows from \eqref{commutation}.

\end{proof}

In the following calculation, we often use a normal coordinate system near a point to diagonalize the Hessian of $u$. Thus at this point,
we can assume that for each $i, j$,
\begin{equation}\label{diagonal}
\sigma_{ij}=\delta_{ij}, u_{;ij}=\lambda_i \delta_{ij},G_{ij}=(1+\lambda_i^2) \delta_{ij},(G^{-1})^{ij}=\frac{\delta_{ij}}{(1+\lambda_i^2)}
\end{equation} where $\lambda_i, i=1\cdots n$ are the eigenvalues of $u_{;ij}$.

In the case when  the sectional curvatures $\sigma_\Sigma$ has a lower bound $c$, we have the following proposition.
\begin{pro}\label{eta-ineq} Suppose $u$ is a solution of the evolution equation \eqref{eq_u-1} on a Riemannian manifold $(\Sigma, \sigma)$.
If the sectional curvatures $\sigma_\Sigma$ of $(\Sigma,\sigma)$ satisfy $\sigma_\Sigma \ge c$
for
$c\in\real{}$, then at a point where \eqref{diagonal} holds true, we have
$$\frac{\partial}{\partial t}\vartheta- (G^{-1})^{ij} \vartheta_{;ij} \le-2\sum_{i=1}^n\frac{\lambda_i^2}{1+\lambda_i^2}
-2c \sum_{p=1}^n \frac{1}{1+\lambda_p^2} (\sum_{i\not=p} u_{;i}^2)
\,.$$
In particular, if $\Sigma$ is compact  and $c \geq 0$, then for $t\in [0, T)$,
$$\vartheta\le \max_{t=0}\vartheta$$
\end{pro}
\begin{proof}
We simply the right hand side of \eqref{eq_eta2}  at a point where \eqref{diagonal} holds,
\[2\sigma^{ij}(G^{-1})^{pq}C^l_{pqi}u_{;l}u_{;j}=-2\sum_{p=1}^n \frac{1}{1+\lambda_p^2}(\sum_{l, i} C_{lpip} u_{;l} u_{;i}).\]

We may assume that the coordinate at this point is chosen so that $C_{lpip}=0$ if $l\not =i$. Note that $C_{ipip}, p\not= i$ is the sectional curvature spanned by the $i$ and $p$ directions, and thus by assumption,
\begin{equation}
\begin{split} & 2\sigma^{ij}(G^{-1})^{pq}C^l_{pqi}u_{;l}u_{;j} \leq -2c \sum_{p=1}^n \frac{1}{1+\lambda_p^2} (\sum_{i\not=p} u_{;i}^2)
\end{split}
\end{equation}
and
$$\sigma^{ij}(G^{-1})^{pq}u_{;ip}u_{;jq}=\sum_{i=1}^n\frac{\lambda_i^2}{1+\lambda_i^2}.$$

The last statement follows from the maximum principle.
\end{proof}
Consider the evolution equation of $\rho$, where
\begin{equation}\label{eq_rho} \rho=\frac{1}{2}\ln \det G_{ij}-\frac{1}{2}\ln\det {\sigma_{ij}}=\frac{1}{2}\ln \det (\sigma_{ij}+u_{;ik} \sigma^{kl} u_{;lj})-\frac{1}{2}\ln\det {\sigma_{ij}}.\end{equation}

\begin{lem} Suppose $u$ is a solution of the evolution equation \eqref{eq_u-1} on a Riemannian manifold $(\Sigma, \sigma)$, then $\rho$ defined in equation \eqref{eq_rho} satisfies the following evolution equation:
\begin{equation}\label{eq_rho2}
\begin{split}
&\frac{\partial \rho}{\partial t}-(G^{-1})^{kl}{\rho}_{;kl}\\
&=(G^{-1})^{ij} (G^{-1})^{pq} u^{\,k}_{;\,\,j} (\Xi_1)_{pqik} \\
&-(G^{-1})^{kl}(G^{-1})^{pq}u_{;prk} \sigma^{rs}u_{;sql}-\frac{1}{2} (G^{-1})^{kl} G_{pq;k} (G^{-1})^{pq}_{;\,\,\,l}+(G^{-1})^{pq}_{;\,\,\,k} u_{;pqi}u^{\,k}_{;\,\,j} (G^{-1})^{ij},\\
\end{split}
\end{equation} where \[ (\Xi_1)_{pqik}=u_{;lk} C^l_{\,\,\,pqi}+u_{;l}C^l_{\,\,\,pqi;k}+u_{;lp} C^l_{\,\,\,iqk}+u_{;il} C^l_{\,\,\,pqk}+u_{;lq} C^l_{\,\,\,ipk}+u_{;l} C^l_{\,\,\, ipk;q}\]
\end{lem}

\begin{proof} We first verify the following two identities:
\begin{equation}\begin{split}\label{eq_rho3}\frac{\partial \rho}{\partial t}&=\left[(G^{-1})^{pq} u_{;pqik}+(G^{-1})^{pq}_{;k} u_{;pqi}\right] u^{\,k}_{;\,\,j} (G^{-1})^{ij}\\
\rho_{;kl}&=u_{;prkl} u^{\,r}_{;\,\,q} (G^{-1})^{pq}+ u_{;prk} u^{\,r}_{;\,\,ql} (G^{-1})^{pq} +\frac{1}{2} G_{pq;k} (G^{-1})^{pq}_{;l}.\end{split}\end{equation}

By the definition of $\rho$ and $G_{ij}$, after symmetrization we get
\[\frac{\partial \rho}{\partial t}=(\frac{\partial u_{;ik}}{\partial t})   u^{\,k}_{;\,\,j} (G^{-1})^{ij}.\]

Differentiating \eqref{eq_du}, we obtain \[\frac{\partial u_{;ij}}{\partial t}=\theta_{;ij}.\] Recall from \eqref{eq_dtheta} $\theta_{;i}=(G^{-1})^{pq} u_{;pqi}$ and differentiate this equation one more time, we derive
\[\theta_{;ik}=(G^{-1})^{pq}_{;k} u_{;pqi}+(G^{-1})^{pq} u_{;pqik}.\] This gives the first formula in  \eqref{eq_rho3}. On the other hand,
\[\rho_{;k}=\frac{1}{2} (G_{pq})_{;k} (G^{-1})^{pq}\text{ and }G_{pq;k}=u_{;prk} u^{\,r}_{;\,\,q}+
u^{\,r}_{;\,\,p}u_{;rqk}.\] Differentiating one more time, we obtain the second formula in \eqref{eq_rho3}. To this end, it suffices to compute
\[(G^{-1})^{pq}u_{;pqik} u^{\,k}_{;\,\,j} (G^{-1})^{ij}-(G^{-1})^{kl} u_{;prkl} u^{\, r}_{;\,\,q} (G^{-1})^{pq}=(G^{-1})^{ij} (G^{-1})^{pq}(u_{;pqik}-u_{;ikpq})u^{\,k}_{;\,\,j}.\]
We write \[u_{;pqik}=(u_{;pqi}-u_{;piq})_{;k}+(u_{;ipqk}-u_{;ipkq})+(u_{;ipk}-u_{;ikp})_{;q}+u_{;ikpq}.\] Therefore, by the commutation formula for curvature tensor in equation \eqref{commutation}, we obtain
\[\begin{split}u_{;pqik}&=u_{;lk} C^l_{\,\,pqi}+u_{;l}C^l_{\,\,pqi;k}+u_{; lp}C^l_{\,\,iqk}+u_{;il} C^l_{\,\,pqk} +u_{;lq}C^l_{\,\,ipk}+u_{;l} C^l_{\,\,ipk;q}+u_{;ikpq}\\
&=(\Xi_1)_{pqik}+u_{;ikpq}.\end{split}\]
\end{proof}
We simplify the right hand side of equation \eqref{eq_rho2} at a point  using \eqref{diagonal}.
\begin{pro} Suppose $u$ is a solution of the evolution equation \eqref{eq_u-1} on a Riemannian manifold $(\Sigma, \sigma)$, then $\rho$ defined in equation \eqref{eq_rho} satisfies the following equation at a point where \eqref{diagonal} holds true:
\begin{equation}\label{eq_rho_sphere}
\begin{split}
&\frac{\partial \rho}{\partial t}-(G^{-1})^{kl}{\rho}_{;kl}\\
&= \sum_{p, q, k} \frac{-1+\lambda_p\lambda_q-\lambda_k(\lambda_p+\lambda_q)}{(1+\lambda_p^2)(1+\lambda_q^2)(1+\lambda_k^2)} u_{;pqk}^2- \sum_{p<k}\frac{2(\lambda_p-\lambda_k)^2}{(1+\lambda_p^2)(1+\lambda_k^2)} C^p_{\,\,\, kpk}\\
&+\sum_{p,k, l}\frac{(\lambda_k-\lambda_p)}{(1+\lambda_k^2)(1+\lambda_p^2)}u_{;l} C^l_{\,\,\, kpk;p}.
\end{split} \end{equation}
\end{pro}

\begin{proof}
We compute those  terms that involve the third covariant derivatives $u_{;pqk}$ first. For any fixed indexes $p, q, k$, we derive
\[(G^{-1})^{pq}_{;k}=-(G^{-1})^{pr}(u_{;rmk} u^{\,m}_{;\,\,s}+u^{\,m}_{;\,\,r} u_{;msk})(G^{-1})^{sq}=-\frac{\lambda_p+\lambda_q}{(1+\lambda_p^2)(1+\lambda_q^2)} u_{;pqk}\] and
\[G_{pq;k}=(\lambda_p+\lambda_q) u_{;pqk}.\]

Therefore,
\[\sum_{p, q, k, l}-\frac{1}{2}(G^{-1})^{kl} G_{pq;k} (G^{-1})^{pq}_{;l}=\sum_{p, q, k}\frac{1}{2} \frac{(\lambda_p+\lambda_q)^2}{(1+\lambda_p^2)(1+\lambda_q^2)(1+\lambda_k^2)} u_{;pqk}^2 \] and
\[\sum_{p, q,i, j, k}(G^{-1})^{pq}_{;k} u_{;pqi} u^{\,k}_{;\,\,j} (G^{-1})^{ij}=-\sum_{p, q, k}\frac{(\lambda_p+\lambda_q)\lambda_k}{(1+\lambda_p^2)(1+\lambda_q^2)(1+\lambda_k^2)} u_{;pqk}^2.\]

On the other hand,
\[-\sum_{p, q, k, l, r, s} (G^{-1})^{kl} (G^{-1})^{pq} u_{;prk} \sigma^{rs} u_{;sql}=\sum_{p, q, k}\frac{-1}{(1+\lambda_p^2)(1+\lambda_k^2)} u_{;pqk}^2.\] Adding up the last three terms and symmetrizing the indexes $p$ and $q$, we obtain
\[\sum_{p, q, k}\frac{-1+\lambda_p\lambda_q-\lambda_k(\lambda_p+\lambda_q)}{(1+\lambda_p^2)(1+\lambda_q^2)(1+\lambda_k^2)} u_{;pqk}^2.\] When the base manifold is flat the indexes $q$ and $k$ are symmetric and this term is
\[-\sum_{p, q, k}\frac{(1+\lambda_p\lambda_q)}{(1+\lambda_p^2)(1+\lambda_q^2)(1+\lambda_k^2)} u_{;pqk}^2.\] This recovers the equation in \cite{sw1}.

Now we turn to ambient curvature term, first of all we observe that for fixed indexes $i, k, p, q$,
\[\sum_j (G^{-1})^{ij} (G^{-1})^{pq}u^{\,k}_{;\,\,j}=\frac{\lambda_k}{(1+\lambda_k^2){(1+\lambda_p^2)}}\delta_{pq}\delta_{ik}.\] Therefore the ambient curvature term becomes
\[\frac{\lambda_k}{(1+\lambda_k^2)(1+\lambda_p^2)}(2\lambda_k C^k_{\,\,\,ppk}+2\lambda_p C^p_{\,\,\,kpk}+\sum_l u_{;l } C^l_{\,\,\,ppk;k}+u_{;l}C^l_{\,\,\, kpk;p}).\] Symmetrizing $k$ and $p$ using the symmetry of the curvature operator, we obtain
\[\sum_{p, k} \frac{-(\lambda_p-\lambda_k)^2}{(1+\lambda_p^2)(1+\lambda_k^2)} C^p_{\,\,\, kpk}+\sum_{p, k, l} \frac{(\lambda_k-\lambda_p)}{(1+\lambda_k^2)(1+\lambda_p^2)}u_{;l} C^l_{\,\,\, kpk;p}.\] The first term is non-positive if the sectional curvature of $g$ is.

\end{proof}

Note that, in view of the definition of $\rho$ \eqref{eq_rho}, at a point where \eqref{diagonal} holds, \[\rho=\frac{1}{2} \ln \left[\prod_{i=1}^n (1+\lambda_i^2)\right].\]
If $\rho$ is close to $0$, $\lambda_i$ is also small. In this case, the first term on the right hand side of \eqref{eq_rho_sphere}  is negative and we can show that $\rho$ being close to $0$ is preserved along the flow.

Next, we compute the evolution equation of the third derivatives of $u$, which corresponds to the second fundamental forms of the Lagrangian submanifold defined by $du$.

Let \begin{equation}\label{eq_Theta}\Theta^2 = (G^{-1})^{ip}(G^{-1})^{jq}(G^{-1})^{kr}u_{;ijk}u_{;pqr}\end{equation}
 and \begin{equation}\label{eq_Upsilon}\Upsilon^2=(G^{-1})^{m s}(G^{-1})^{ip}(G^{-1})^{jq}(G^{-1})^{kr}u_{;ijkm}u_{;pqrs}.\end{equation}

 \begin{lem} Suppose $u$ is a solution of the evolution equation \eqref{eq_u-1} on a Riemannian manifold $(\Sigma, \sigma)$.
If the curvature tensor of $\Sigma$ is parallel, $\Theta^2$ defined in \eqref{eq_Theta} evolves by the following equation:
\begin{equation}\label{master_eq}
\begin{split} & \ \frac{\partial }{\partial t} \Theta^2-(G^{-1})^{m s} (\Theta^2)_{;ms} \\
=  & -2\Upsilon^2+\ 2 (G^{-1})^{ip}(G^{-1})^{jq}(G^{-1})^{kr}(G^{-1})^{m s}_{;jk}u_{ms i}u_{;pqr}\\
- &  (G^{-1})^{m s}  \left[ 2  (G^{-1})^{ip}_{;ms}(G^{-1})^{jq}(G^{-1})^{kr}
+(G^{-1})^{ip}(G^{-1})^{jq}(G^{-1})^{kr}_{;ms}\right]u_{;ijk}u_{;pqr}
+ \mbox{I}+ \mbox{II} + \mbox{III},\\
\end{split}
\end{equation}
where
\begin{equation}\label{I}\mbox{I} = \left[ 2\frac{\partial (G^{-1})^{ip}}{\partial t}(G^{-1})^{jq}(G^{-1})^{kr}+ (G^{-1})^{ip}(G^{-1})^{jq}\frac{\partial (G^{-1})^{kr}}{\partial t}\right] u_{;ijk}u_{;pqr}\;,\end{equation}
\begin{equation}\label{II}
\begin{split}
 \ \mbox{II}
= & \ -(G^{-1})^{m s} \left[ 2  (G^{-1})^{ip}_{;m}(G^{-1})^{jq}_{;s}(G^{-1})^{kr}+ 4 (G^{-1})^{ip}_{;m}(G^{-1})^{jq}(G^{-1})^{kr}_{;s}\right]
u_{;ijk}u_{;pqr}\\
 & \   -(G^{-1})^{m s} \left[ 8(G^{-1})^{ip}_{;s}(G^{-1})^{jq}(G^{-1})^{kr} + 4(G^{-1})^{ip}(G^{-1})^{jq}(G^{-1})^{kr}_{;s} \right] u_{;ijkm}u_{;pqr},
\end{split}
\end{equation}
and
\begin{equation}
\begin{split}
 \ \mbox{III}
=& \ 2(G^{-1})^{ip}(G^{-1})^{jq}(G^{-1})^{kr}u_{;pqr} \Big[(G^{-1})^{m s}_{;j}u_{;ms ik}+
(G^{-1})^{m s}_{;k}u_{;ms ij}\\
+ & \ (G^{-1})^{m s}(2u_{;ljs }C^l_{im k}
+2u_{;ils }C^l_{jm k}
+2u_{;lm k}C^l_{is j}+u_{; ilk}C^l_{ms j}
+u_{;ijl}C^l_{ms k}+u_{;ljk}C^l_{m  s i})
\Big].\\
\end{split}
\end{equation}

\end{lem}
\begin{proof}

Recall that $u_{;ijk}$ is symmetric in the $i,j$ indexes.
A straightforward calculation using this symmetry gives
\begin{equation}\label{ddt_Theta^2}
\begin{split}
&  \frac{\partial }{\partial t} \Theta^2\\
=& 2 \frac{\partial (G^{-1})^{ip}}{\partial t}(G^{-1})^{jq}(G^{-1})^{kr}u_{;ijk}u_{;pqr}+ (G^{-1})^{ip}(G^{-1})^{jq}\frac{\partial (G^{-1})^{kr}}{\partial t}u_{;ijk}u_{;pqr}\\
+& 2 (G^{-1})^{ip}(G^{-1})^{jq}(G^{-1})^{kr}\theta_{;ijk}u_{;pqr}
\end{split}
\end{equation}
and
\begin{equation}\label{Laplace_Theta^2}
\begin{split}
& \ (G^{-1})^{m s} (\Theta^2)_{;ms}\\
= & \ 2 (G^{-1})^{m s}(G^{-1})^{ip}(G^{-1})^{jq}(G^{-1})^{kr}u_{;ijkms}u_{;pqr}\\
+&2 (G^{-1})^{m s}(G^{-1})^{ip}(G^{-1})^{jq}(G^{-1})^{kr}u_{;ijkm}u_{;pqrs}\\
+ & \ (G^{-1})^{m s} \left[ 2  (G^{-1})^{ip}_{;ms}(G^{-1})^{jq}(G^{-1})^{kr}
+(G^{-1})^{ip}(G^{-1})^{jq}(G^{-1})^{kr}_{;ms}\right]u_{;ijk}u_{;pqr}\\
+ & \ (G^{-1})^{m s} \left[ 2  (G^{-1})^{ip}_{;m}(G^{-1})^{jq}_{;s}(G^{-1})^{kr}+ 4 (G^{-1})^{ip}_{;m}(G^{-1})^{jq}(G^{-1})^{kr}_{;s}\right]
u_{;ijk}u_{;pqr}
\\
+ & \ (G^{-1})^{m s}\left[8(G^{-1})^{ip}_{;s}(G^{-1})^{jq}(G^{-1})^{kr} + 4(G^{-1})^{ip}(G^{-1})^{jq}(G^{-1})^{kr}_{;s} \right]u_{;ijkm}u_{;pqr}.\\
\end{split}
\end{equation}

Subtracting \eqref{Laplace_Theta^2} from \eqref{ddt_Theta^2} and regrouping terms, we derive
\begin{equation}
\begin{split}
& \ \frac{\partial }{\partial t} \Theta^2-(G^{-1})^{m s} (\Theta^2)_{;ms} \\
=  & -2\Upsilon^2+ \ 2 (G^{-1})^{ip}(G^{-1})^{jq}(G^{-1})^{kr}(\theta_{;ijk}-(G^{-1})^{m s}u_{;ijkms} )u_{;pqr}\\
 & -(G^{-1})^{m s}  \left[ 2  (G^{-1})^{ip}_{;ms}(G^{-1})^{jq}(G^{-1})^{kr}
+(G^{-1})^{ip}(G^{-1})^{jq}(G^{-1})^{kr}_{;ms}\right]u_{;ijk}u_{;pqr}
+ \mbox{I}+ \mbox{II}. \\
\end{split}
\end{equation}

Note that
\begin{equation}\label{theta_ijk}
\begin{split}
\theta_{;ijk}=& [(G^{-1})^{ms}u_{;ms i}]_{;jk}\\
=& (G^{-1})^{ms}u_{;ms ijk}+(G^{-1})^{m s}_{;jk}u_{;ms i}+ (G^{-1})^{m s}_{;j}u_{;ms ik}+
(G^{-1})^{m s}_{;k}u_{;ms ij}.\\
\end{split}
\end{equation}

Using  commutation formulae in equation \eqref{commutation}, we obtain, under the assumption of parallel curvature tensor,
\begin{eqnarray*}
& &u_{;ms ijk} \\
&=&   u_{;ijm ks   } + u_{;ljm}C^l_{is k}  + u_{;ilm}C^l_{js k}+  u_{;ij\l}C^l_{ms k}\\
&+& u_{;ls k}C^l_{im j}
+u_{;lm k}C^l_{ i s j }+u_{;ilk}C^l_{ m s j }
+u_{;ljk}C^l_{m  s i}. \\
\end{eqnarray*}
In addition,
\[ u_{;ijm ks  }
= u_{;ijkms   }
+u_{;lj s}C^l_{im k}
+u_{;il s}C^l_{jm k}.\] Thus
\begin{equation}\begin{split}\label{commute_4}
 &u_{;ms ijk}- u_{;ijkms   }\\
=&   u_{;lj s}C^l_{im k}
+u_{;il s}C^l_{jm k}
+ u_{;ljm}C^l_{is k}  + u_{;ilm}C^l_{js k}+  u_{;ij\l}C^l_{ms k}\\
+& u_{;ls k}C^l_{im j}
+u_{;lm k}C^l_{ i s j } + u_{;ilk}C^l_{ m s j }
+u_{;ljk}C^l_{m  s i}. \\
\end{split}\end{equation}

Combing \eqref{theta_ijk} and \eqref{commute_4}, we obtain
\begin{equation}
\begin{split}
& \ (G^{-1})^{ip}(G^{-1})^{jq}(G^{-1})^{kr}(\theta_{;ijk}-(G^{-1})^{m s}u_{;ijkms})u_{;pqr} \\
= & \ (G^{-1})^{ip}(G^{-1})^{jq}(G^{-1})^{kr}\Big[  (G^{-1})^{m s}_{;jk}u_{;ms i}+ (G^{-1})^{m s}_{;j}u_{;ms ik}+
(G^{-1})^{m s}_{;k}u_{;ms ij}\\
+ & \ (G^{-1})^{m s}(2u_{;ljs }C^l_{im k}
+2u_{;ils }C^l_{jm k}
+2u_{;lm k}C^l_{is j}+u_{; ilk}C^l_{ms j}
+u_{;ijl}C^l_{ms k}+u_{;ljk}C^l_{m  s i})\Big] u_{;pqr}.
\end{split}
\end{equation}
\end{proof}

 In the rest of this section, we estimate the right hand side of \eqref{master_eq}. We introduce another
 geometric quantity $\Lambda\geq 0$ to measure the Hessian of $u$:
 \begin{equation}\label{Lambda}
 \Lambda^2=\sigma^{ij}\sigma^{kl} u_{;ik}u_{;jl}
 \end{equation}
 We prove the following differential inequality.
\begin{pro} Suppose $u$ is a solution of the evolution equation \eqref{eq_u-1} on a Riemannian manifold $(\Sigma, \sigma)$.
If the curvature tensor of $\Sigma$ is parallel, $\Theta^2$ defined in \eqref{eq_Theta} satisfies the following inequality:
\begin{equation}\label{diff_ineq}\frac{\partial}{\partial t} \Theta^2-(G^{-1})^{ms}(\Theta^2)_{;ms}\leq
-\Upsilon^2 +
C_1 (1+\Lambda^2) \Theta^4+C_2 \Theta^2,\end{equation} where
$\Lambda$ is defined in \eqref{Lambda}
and $C_1$ and $C_2$  are constants that depend only
on the dimension of $\Sigma$.
\end{pro}

\begin{proof}

From \[G_{ij}=\sigma_{ij}+ u_{;ik}\sigma^{kl}u_{;lj}\;,\] we compute that
\begin{equation}\begin{split}\label{one_der_G-1}
(G^{-1})^{pq}_{;\,\,\,k}
=& -(G^{-1})^{pr}(u_{;rmk} u^m_{;s}+u^m_{;r} u_{;msk})(G^{-1})^{sq}\\
=& -u_{;s}^m u_{;mrk}\Big[(G^{-1})^{pr}(G^{-1})^{sq}+(G^{-1})^{ps}(G^{-1})^{rq}\Big].
\end{split}\end{equation}

Taking one more derivative, we derive:
\begin{equation}\label{3rd-der-1}
\begin{split}
 \ (G^{-1})^{pq}_{;\,\,\,kj}
= & \ -{u_{;sj}^m} u_{;mrk}\Big[(G^{-1})^{pr}(G^{-1})^{sq}+(G^{-1})^{ps}(G^{-1})^{rq}\Big]\\
- & \ u_{;s}^m \Big\{ u_{;mrk}\Big[(G^{-1})^{pr}(G^{-1})^{sq}+(G^{-1})^{ps}(G^{-1})^{rq}\Big]\Big\}_{;j}.
\end{split}
\end{equation}

On the other hand, using $\frac{\partial u_{;ij}}{\partial t}=\theta_{;ij},$ we compute
\[\frac{\partial G_{ij}}{\partial t}
= \theta_{;ik}\sigma^{kl}u_{;lj} +u_{;ik}\sigma^{kl}\theta_{;lj}\;.
\] Differentiating $\theta_{;i}=(G^{-1})^{pq} u_{;pqi}$ one more time gives
\begin{equation}\label{theta-2der}\theta_{;ik}=(G^{-1})^{pq}_{;k} u_{;pqi}+(G^{-1})^{pq} u_{;pqik}.
\end{equation}

Within this section, for any positive integer $i$, $C_i$ denotes a positive constant that depends only on the dimension $n$. At any point where \eqref{diagonal} holds true $\Lambda=\sqrt{\sum_{i=1}^n\lambda_i^2}$ and
\begin{equation}\begin{split} |u^{\,i}_{;\,\,j}|&\leq \Lambda \delta^i_j, \text{  for any} i, j\\
|(G^{-1})^{kl}u_{;lj}|&=|\frac{\lambda_k}{1+\lambda_k^2}\delta^k_j|\leq \delta^k_j  \text{ for any} k, j\end{split}\end{equation}

From \eqref{one_der_G-1} and \eqref{theta-2der}, we have
$$|(G^{-1})^{pq}_{;\,\,\,k}| \leq
2\Lambda |u_{;srk}(G^{-1})^{ps}(G^{-1})^{qr}|,$$
$$|\theta_{;ik}| \leq 2 \Lambda |(G^{-1})^{ps}(G^{-1})^{qr}    u_{;pqi} u_{;srk}|+ |(G^{-1})^{pq}u_{;pqik}|,$$
and $|\frac{\partial G_{ij}}{\partial t}|\leq 2 |\theta_{;ij}|.$
Thus
\begin{equation}
\begin{split} &
| 2\frac{\partial (G^{-1})^{ip}}{\partial t}(G^{-1})^{jq}(G^{-1})^{kr} u_{;ijk}u_{;pqr}|\\
\leq & 4 (G^{-1})^{ir}(G^{-1})^{sp}(G^{-1})^{jq}(G^{-1})^{kr}|\theta_{;rs}| |u_{;ijk}||u_{;pqr}|\\
\leq  & C_1 \Lambda \Theta^4 + C_2 \Theta^2 \Upsilon
\end{split}
\end{equation}
for some constants $C_1, C_2$ depending only on $n$.
Similarly,
$$(G^{-1})^{ip}(G^{-1})^{jq}\frac{\partial (G^{-1})^{kr}}{\partial t}u_{;ijk}u_{;pqr}
\leq C_1 \Lambda \Theta^4 + C_2 \Theta^2 \Upsilon.$$

Thus $|I| \leq C_1 \Lambda \Theta^4 + C_2 \Theta^2 \Upsilon$.

Similarly, we have
$$|II| \leq C_3 \Lambda^2 \Theta^4 + C_4 \Lambda\Theta^2 \Upsilon.$$
and
$$|III| \leq C_5 \Lambda \Theta^2\Upsilon  + C_6 \Theta^2.$$
Using $|(G^{-1})^{pq}_{;\,\,\,k}| \leq
2 \Lambda |u_{;srk}(G^{-1})^{ps}(G^{-1})^{qr}|$, $|u_{;s}^m|\leq \Lambda\delta_{s}^m$
and \eqref{3rd-der-1}, we have
\begin{equation}
\begin{split}
 & |2 (G^{-1})^{ip}(G^{-1})^{jq}(G^{-1})^{kr}(G^{-1})^{m s}_{;jk}u_{ms i}u_{;pqr}\\
- &  (G^{-1})^{m s}  \left[ 2  (G^{-1})^{ip}_{;ms}(G^{-1})^{jq}(G^{-1})^{kr}
+(G^{-1})^{ip}(G^{-1})^{jq}(G^{-1})^{kr}_{;ms}\right]u_{;ijk}u_{;pqr}|\\
\leq & C_7(1+\Lambda^2)\Theta^4 +C_7 \Lambda \Upsilon\Theta^2
\end{split}
\end{equation}

The right hand side of \eqref{master_eq} can thus be bounded from above by
\[-2\Upsilon^2+ C_{14}(1+\Lambda^2) \Theta^4+ C_{15}(1+ \Lambda)\Upsilon \Theta^2+ C_{16}\Theta^2.\]
The claim \eqref{diff_ineq} follows from this and the Cauchy-Schwarz inequality.
\end{proof}

\section{Proof of Theorem 3}

We give the precise statement of Theorem 3:

\noindent {\bf Theorem 3}
{\it When $(\Sigma, \sigma)$ is a standard round sphere of constant sectional curvature, the zero section in $T^*\Sigma$ is stable under the generalized Lagrangian mean curvature flow. Suppose a Lagrangian submanifold $M_0$ is the graph of $du$ for a smooth
function $u$ on $\Sigma$ and let $\lambda_i$ be the eigenvalues of the Hessian of $u$ with respect to $\sigma$. There exists a constant $\epsilon$ depending only on $n$ and the curvature of $\Sigma$ such that if  $\prod_{i=1}^n(1+\lambda_i^2)\leq 1+\epsilon$,  then generalized Lagrangian mean curvature flow of $M_0$ exists smoothly for all time,
and converges to the zero section smoothly at infinity.}

\begin{proof}
Let $\chi = \frac{\det G_{ij}}{\det {\sigma_{ij}}}= \prod_{i=1}^n(1+\lambda_i^2)$.
From the condition $\chi \leq 1+\epsilon$,
we have  $\Lambda^2= \sum_i\lambda_i^2 \leq \epsilon$ and $\lambda_i\lambda_j\leq \epsilon$ for $1\leq i, j \leq n$.
Since the section curvature of $\sigma$ is positive and the curvature tensor is parallel,  the evolution equation of $\rho$ in  \eqref{eq_rho2} implies
hat the condition $\chi\leq 1+\epsilon$ is preserved by the generalized Lagrangian mean curvature flow if $3\epsilon \leq 1$.

In particular, by assuming $3\epsilon \leq 1$, we obtain the following differential inequality along the flow:
\[\frac{\partial \rho}{\partial t}-(G^{-1})^{kl}\rho_{;kl} \leq (-1+3 \epsilon)\Theta^2.\]
In the following calculation, we denote $\nabla_G f\cdot \nabla_G g=(G^{-1})^{kl} f_{;k} g_{;l}$ and
$|\nabla_G f|^2=(G^{-1})^{kl} f_{;k} f_{;l}$ for functions $f$ and $g$ defined on $\Sigma$.
With $\rho= \frac{1}{2}\ln \chi$, the last inequality can be turned into a differential inequality of $\chi$:
\begin{equation}\label{chi-evo}\frac{\partial \chi}{\partial t}-(G^{-1})^{kl}\chi_{;kl}\leq 2(-1+3 \epsilon)\chi\Theta^2- \frac{|\nabla_G \chi|^2 }{\chi}.
\end{equation}
Since $\Lambda^2\leq \epsilon$,
we  have
\begin{equation}\label{theta-evo}\frac{\partial}{\partial t} \Theta^2-(G^{-1})^{kl}(\Theta^2)_{;kl}
\leq -\Upsilon^2+ C_{1}(1+\epsilon) \Theta^4+ C_{2}\Theta^2\end{equation}
from \eqref{diff_ineq}.
Let $p$ be a positive number to be determined, we compute:
\begin{equation*}
\begin{split}
 \ & \frac{\partial }{\partial t}\Big(\chi^p\Theta^2\Big)-(G^{-1})^{kl}(\chi^p\Theta^2) _{;kl}\\
= & \ p\chi^{p-1}\Theta^2\Big(\frac{\partial \chi}{\partial t}-(G^{-1})^{kl}\chi_{;kl}\Big)
+\chi^p \Big(\frac{\partial}{\partial t} \Theta^2-(G^{-1})^{kl}(\Theta^2)_{;kl}\Big) \\
- & \ p(p-1)\chi^{p-2}\Theta^2|\nabla_G\chi|^2-2p\chi^{p-1}\nabla_G\chi \cdot \nabla_G(\Theta^2).
\end{split}
\end{equation*}
Using \eqref{chi-evo} and \eqref{theta-evo} in the above equation, we obtain
\begin{equation*}
\begin{split}
 \ & \frac{\partial }{\partial t}\Big(\chi^p\Theta^2\Big)-(G^{-1})^{kl}(\chi^p\Theta^2) _{;kl}\\
\leq & \ 2(-1+3 \epsilon)p\chi^{p}\Theta^4- p^2\chi^{p-2}\Theta^2|\nabla_G \chi|^2-2p\chi^{p-1}\nabla_G\chi \cdot \nabla_G(\Theta^2)\\
+ &\chi^p \Big(-\Upsilon^2+ C_{1}(1+\epsilon)  \Theta^4+ C_{2}\Theta^2\Big) \\
\leq & \ -2p \nabla_G (\chi^p\Theta^2) \cdot \nabla_G\ln \chi+ p^2\chi^{p-2}\Theta^2|\nabla_G\chi|^2\\
 + &  \Big(2(-1+3 \epsilon)p +C_1 (1+\epsilon) \Big)\chi^{p}\Theta^4 + C_2 \chi^p\Theta^2.
\end{split}
\end{equation*}
Note that we used $$ \nabla_G (\chi^p\Theta^2) \cdot \nabla_G\ln \chi= p\chi^{p-2}\Theta^2 |\nabla_G \chi|^2+ \chi^{p-1}\nabla_G\chi \cdot \nabla_G(\Theta^2).$$
Recall that \[\rho_{;k}=\frac{1}{2} (G_{ij})_{;k} (G^{-1})^{ij}
=\sum_i\frac{\lambda_i u_{;iik}}{1+\lambda_i^2}\] and
\[|\nabla_G \rho|^2=\sum_{i,j,k}\frac{\lambda_i\lambda_ju_{;iik}u_{;jjk} }{(1+\lambda_i^2)(1+\lambda_j^2)(1+\lambda_k^2)}
\leq  \frac{\epsilon}{2}\sum_{i,j,k}\frac{u_{;iik}^2+u_{;jjk}^2 }{(1+\lambda_i^2)(1+\lambda_j^2)(1+\lambda_k^2)}
\leq \epsilon \Theta^2.
\]
Thus $p^2\chi^{p-2}\Theta^2|\nabla_G\chi|^2=4p^2\chi^{p}\Theta^2|\nabla_G\rho|^2\leq 4p^2\epsilon^2\chi^{2p}\Theta^4$
where we have also used the fact that $1 \leq \chi$ .
This implies that
\begin{equation*}
\begin{split}
 \ & \frac{\partial }{\partial t}\Big(\chi^p\Theta^2\Big)-(G^{-1})^{kl}(\chi^p\Theta^2) _{;kl}\\
\leq & \ -2p \nabla_G (\chi^p\Theta^2) \cdot \nabla_G\ln \chi+ \Big(4p^2\epsilon^2+2(-1+3 \epsilon)p +C_1(1+\epsilon)  \Big)\chi^{2p}\Theta^4 + C_2 \chi^p\Theta^2.
\end{split}
\end{equation*}
Choose $\epsilon$ small enough so that  $(-1+3 \epsilon)^2-4C_1\epsilon^2(1+\epsilon)>0$
and $1-3 \epsilon > 0$.
Then we can find $p > 0$  so that $4p^2\epsilon^2+2(-1+3 \epsilon)p +C_1(1+\epsilon)$  is negative.
The maximum principle implies that $\chi^p\Theta^2$ is uniformly bounded.
Hence $\Theta^2$ is unformly bounded based on the fact that $\chi \geq 1$.
Standard arguments imply that the higher order derivatives of $u$ are also bounded.
This proves the long time existence and convergence of the generalized Lagrangian mean curvature flow.
Using Proposition \ref{eta-ineq}, $c=1$ and $\sum_i\lambda_i^2 \leq \epsilon$, we have
$$\frac{\partial}{\partial t}\vartheta\le (G^{-1})^{ij} \vartheta_{;ij}
- \frac{2c(n-1)}{1+\epsilon^2}\vartheta
$$ and $\vartheta \leq (\max_{t=0}\vartheta)\cdot e^{\frac{-2(n-1)t}{1+\epsilon^2}}$ or $\vartheta=\sigma^{ij} u_{;i} u_{;j}$ is sub-exponential decay.
This shows that the section $du$ converges to the zero section.
\end{proof}

Finally, we remark that the stability theorem (Theorem 3) holds true when the sphere is replaced by a compact Riemannian manifold of positive sectional curvature. Lemma 6.4 needs to be modified to accommodate the covariant derivatives of the curvature tensor. However, the contribution is of lower order, and Theorem 3 still holds, except that the constant $\epsilon$ depends on the covariant
derivatives of the curvature as well.

\end{document}